\documentclass{article}
\usepackage{amssymb,amsfonts,amsmath,latexsym,verbatim,amscd,amsthm,graphicx,epsfig,color}

\newcommand{\equ}[1]{(\ref{#1})}

\newcommand{\be}{\begin{equation}}
\newcommand{\ee}{\end{equation}}
\newcommand{\bes}{\begin{equation*}}
\newcommand{\ees}{\end{equation*}}
\newcommand{\RR}{{\mathbb R}}
\newcommand{\HH}{{\mathbb H}}
\newcommand{\rr}{{\mathbb R}}
\newcommand{\ms}{{\mathbb S}}
\newcommand{\R}{{\mathbb R}}
\newcommand{\hh}{{\mathbb H}}

 \newtheorem{teo}{Theorem}[section]
 \newtheorem{definition}[teo]{Definition}
 \newtheorem{lemma}[teo]{Lemma}
 
\newtheorem{prop}[teo]{Proposition}
 \newtheorem{rema}[teo]{Remark}
 \newtheorem{cor}[teo]{Corollary}
\numberwithin{equation}{section}

  \newcommand{\calC}{{\mathcal C}}

 \newcommand{\calH}{{\mathcal H}}

\newcommand{\del}{\partial}
\newcommand{\sech}{\operatorname{sech}}
\newcommand{\e}{\epsilon}
\newcommand{\lp}{\left(}
\newcommand{\rp}{\right)}
\newcommand{\abs}[1]{\lvert#1\rvert}
\DeclareMathOperator{\divergence}{div}

\begin{document}

\title{\bf Layer solutions for the fractional \\Laplacian on hyperbolic space: existence, uniqueness and qualitative properties}
 \author{Gonz\'alez, Mar\'ia del Mar\footnote{Universitat Polit\`ecnica de Catalunya, ETSEIB-MA1, Av. Diagonal 647, 08028 Barcelona, Spain. Supported by grants MINECO MTM2011-27739-C04-01 and GENCAT 2009SGR345.}\\
S\'aez, Mariel\footnote{P. Universidad
 Cat\'olica de Chile. Supported by
 grants  Fondecyt regular 1110048  and proyecto Anillo ACT-125, CAPDE.}\\
 Sire, Yannick\footnote{Universit\'e Aix-Marseille, France. Supported by ANR projects ``PREFERED" and ``HAB".}}

\maketitle

\begin{abstract}
We investigate the equation
\begin{equation*}(-\Delta_{\mathbb H^n})^{\gamma} w=f(w)\quad\text{in }\mathbb H^{n},\end{equation*}
where $(-\Delta_{\mathbb H^n})^\gamma$ corresponds to the fractional Laplacian on hyperbolic space for $\gamma \in (0,1)$ and $f$ is a smooth nonlinearity that typically comes from a double well potential. We prove the existence of heteroclinic connections in the following sense; a so-called layer solution is a smooth solution of the previous equation converging to $\pm 1$ at any point of the two hemispheres $S_\pm \subset \partial_\infty \mathbb H^n$ and which is strictly increasing with respect to the signed distance to a totally geodesic hyperplane $\Pi.$ We prove that under additional conditions on  the nonlinearity uniqueness holds up to isometry. Then we provide several symmetry results and qualitative properties of the layer solutions. Finally, we consider the multilayer case, at least when $\gamma$ is close to one.
\end{abstract}

\section{Introduction and statement of the results}

We consider the following semilinear problem on hyperbolic space
\begin{equation}\label{initial-problem}(-\Delta_{\mathbb H^n})^{\gamma} w=f(w)\quad\text{in }\mathbb H^{n},\end{equation}
where $f:\mathbb R\to\mathbb R$, $\gamma\in(0,1)$ and $(-\Delta_{\mathbb H^n})^\gamma$ corresponds to the fractional Laplacian on hyperbolic space.

The definition of the operator  $(-\Delta_{\mathbb H^n})^\gamma$ may be introduced using standard functional calculus on hyperbolic space $\mathbb H^n$. However, it was shown in \cite{Banica-Gonzalez-Saez} that it may be realized as the Dirichlet-to-Neumann operator for a degenerate elliptic extension problem analogous to the one considered in \cite{Caffarelli-Silvestre} (for the Euclidean case) and \cite{Chang-Gonzalez} (for the manifold case). More precisely, let $u(x,y)$ be the solution of
\begin{equation*}\begin{cases}
\partial_{yy} u+\frac{a}{y}\partial_y u+\Delta_{\mathbb H^n} u = 0 &\quad\hbox{for } (x,y)\in \hh^n\times \rr_+,\\
u(x,0)=w(x) &\quad\hbox{for } x\in \hh^n,
\end{cases}\end{equation*}
then
$$(-\Delta_{\mathbb H^n})^\gamma w=-d_\gamma\lim_{y\to 0}y^a\partial_y u,$$
where $a=1-2\gamma$, $\gamma\in(0,1)$ and the constant is given in \eqref{constant-d}.
Then one is able to replace the original non-local problem \eqref{initial-problem} by the boundary reaction problem
\begin{equation}\label{problem-extension}
\begin{cases}
\partial_{yy} u+\frac{a}{y}\partial_y u+\Delta_{\mathbb H^n} u = 0 &\quad\hbox{for } (x,y)\in \hh^n\times \rr_+,\\
-y^a \partial_y u |_{y=0}=f(u) &\quad \hbox{for } x\in \hh^n,
\end{cases}
\end{equation}
up to a (positive) multiplicative constant in front of the nonlinearity.

Semilinear equations for the standard Laplacian on hyperbolic space have received a lot of attention. See \cite{Mancini-Sandeep,Castorina-Fabbri-Mancini-Sandeep,Bonforte-Gazzola-Grillo-Vazquez,
Berchio-Ferrero-Grillo}, for instance, for power nonlinearities. In this paper the nonlinearity $f$  comes from a double well potential, i.e., $f=-F'$, where $F$ is a scalar function satisfying $F(-1)=F(1)=0$, $F>0$ outside $\pm 1$ and $F'(-1)=F'(1)=0$. Since regularity of $F$ will not be an issue here, we assume that $F\in\mathcal C^\infty$. We also suppose that $F''(-1),F''(1)>0$, although this assumption may be weakened to the requirement that $f$ is non-increasing in $(-1,\tau)\cup (\tau,1)$ for some $\tau\in(0,1)$; we will not consider this generalization here. Our purpose is to study layer solutions for equation \eqref{initial-problem}, or equivalently, for \eqref{problem-extension}.

In the Euclidean case, this problem was thoroughly studied in \cite{Cabre-SolaMorales} for the half Laplacian, and then generalized to every power $\gamma\in(0,1)$ by \cite{Cabre-Sire:I, Cabre-Sire:II}. It turns out that many of their arguments can be easily generalized to the hyperbolic case. However, there are two important differences: the lack of translation invariance (which is replaced by invariance by isometries in hyperbolic space), and the presence of a metric weight that makes the problem non-integrable.

The interesting aspect here is to understand how the presence of curvature affects the results in the flat case. For the standard Laplacian ($\gamma=1$) layer solutions are quite well understood (see \cite{Farina-Sire-Valdinoci:I,Farina-Sire-Valdinoci:II} for positive curvature assumptions, and \cite{Birindelli-Mazzeo,Mazzeo-Saez,Pisante-Ponsiglione} for the hyperbolic space case). In \cite{Birindelli-Mazzeo} the authors show that arbitrary bounded global solutions reduce to functions of one variable if they have asymptotic boundary values on $\mathbb S^{n-1}=\partial_{\infty}\mathbb H^n$ which are invariant under a cohomogeneity one subgroup of the group of isometries on $\mathbb H^n$. Existence of these one-dimensional ODE solutions is also proved. In addition, in \cite{Mazzeo-Saez} multiple-layer solutions are considered. More precisely, they prove that for any collection of widely separated, non-intersecting hyperplanes in $\mathbb H^n$, there is a solution which has nodal set very close to this collection of hyperplanes. See also the related work \cite{Pisante-Ponsiglione} for a Gamma-convergence result on hyperbolic space.

Problem \eqref{initial-problem} models phase transitions, and it is also important in the study of  the conjecture posed by De Giorgi in \cite{De Giorgi}.  The standard De Giorgi conjecture stands for the Allen-Cahn equation (case $\gamma=1$) on Euclidean space. It has been solved for $n=2$ \cite{Gousshoub_Gui}, $n=3$ \cite{Ambrosio-Cabre} and with an additional natural assumption for $4 \leq n \leq 8$ \cite{savin}. In dimension $9$, there is a counter-example \cite{delpino}. The conjecture is about the flatness of level sets of bounded, smooth, increasing in one direction solutions. Here we address the question related to the Gibbons conjecture, i.e. the flatness of level sets of the solution of the Allen-Cahn equation with uniform limits towards $\pm 1$ as one of the variables goes to $\pm \infty$. In the case of the Laplacian in hyperbolic space, it has been addressed in \cite{Birindelli-Mazzeo}. In the Euclidean case, this is solved with different techniques in \cite{hamel}. In the case of the Euclidean fractional Laplacian, the Gibbons conjecture has been solved in \cite{Cabre-SolaMorales, Cabre-Sire:II}.

Our goal is to see how far we can push these results for the fractional Laplacian on hyperbolic space. Let us set up some notations. Let $\Pi$ be a totally geodesic hyperplane in $\mathbb H^n$. Then the usual hyperbolic metric can be written as the warped product of $\RR \times \HH^{n-1}$, with metric $g_{\mathbb H^n} = dt^2 + \cosh^2 t\,g_{\HH^{n-1}}$, where $t$ is the signed distance from $\Pi$. Note that the plane $\Pi$ divides the boundary $\mathbb S^{n-1}=\partial_\infty \mathbb H^n$ into two hemispheres, denoted by $S_+$ and $S_-$. In these coordinates
\[
\Delta_{\HH^n} = \del_{tt} + (n-1)\tanh t\, \del_t + \sech^2 t\, \Delta_{\HH^{n-1}}.
\]

\begin{definition}
We say that $w(x)$ is a layer solution of \eqref{initial-problem} if $w$ is a solution of \eqref{initial-problem} satisfying
\begin{itemize}
\item[i.] the asymptotic boundary condition that $w(x)$ converges to $\pm 1$ at any point in the interior of the two hemispheres $S_{\pm}\subset \partial_\infty\mathbb H^{n}$.

\item[ii. ] $\partial_t w>0$ in $\mathbb H^n$, where $\Pi(=\mathbb H^{n-1})$ is the totally geodesic hyperplane with boundary $S_-\cap S_+$, and $t$ the signed distance function to $\Pi$ in $\mathbb H^n$.
\end{itemize}
Analogously, one may say that $u(x,y)$ is a layer solution of \eqref{problem-extension} if
\begin{itemize}
\item[i. ] $u(x,0)$ converges to $\pm 1$ at any point in the interior of the two hemispheres $S_{\pm}$.

\item[ii. ] $\partial_t u(x,0)>0$ in $\mathbb H^n$.
\end{itemize}
\end{definition}

A one-dimensional solution $w$ of \eqref{initial-problem} is a solution that,  given a fixed  hyperplane $\Pi$, it only depends on the variable $t$. Our first theorem concerns existence and uniqueness:

\begin{teo}\label{thm-one-dimensional} There exists a solution $u(t,y)$ of the extension problem
\begin{equation}\label{1dimsol}\begin{cases}
H_au=0\quad &\text{for }(t,y)\in \mathbb R\times \mathbb R_+,\\
-y^a \partial_y u|_{y=0}=f(u) \quad &\text{for } t\in\mathbb R,
\end{cases}\end{equation}
that satisfies
\begin{align}&\partial_t u(t,0)>  0 \hbox{ for } t\in  \rr,\label{eq3} \\
& u(t,0)\to \pm 1 \hbox{ as } t\to \pm  \infty.\label{eq4}
\end{align}
Here the operator $H_a$ is given by
\be\label{operatorH}H_a=  \del_{yy}+\frac{a}{y}  \del_y+\del_{tt} + (n-1)\tanh t\, \del_t.\ee
This solution is unique up to normalization (e.g. $u(0,0)=0$).
\end{teo}

The uniqueness statement above can be rephrased as ``unique up to stretching'', where
the notion of stretching will be precisely defined in Section \ref{section-uniqueness}. Our second theorem concerns the two-dimensional symmetry of solutions, indeed we show that indeed any layer solution  of \eqref{problem-extension} must be the one found in the previous Theorem \ref{thm-one-dimensional}:

\begin{teo}\label{thm-symmetry}
Let $u$ be a solution of \eqref{problem-extension} satisfying the asymptotic boundary condition that $u(x,y)$ converges uniformly to $\pm 1$ at any point in the interior of the two hemispheres $S_{\pm}\subset \partial_\infty\mathbb H^n$. Let $\Pi\equiv \mathbb H^{n-1}$ be the totally geodesic subspace with boundary $\overline{S_+}\cap \overline{S_-}$, and let $t$ be the signed distance function to this subspace in $\mathbb H^n$.
Then:
\begin{enumerate}
\item $u(x,y)$ depends only on $t,y$, where $t(x)=\textrm{dist}(x,\Pi)$.
\item  we have the monotonicity property $\partial_t u(\cdot, y)>0$, for every $y$.
\item $u$ is unique up to isometries on the $\hh^n$ variable.   \end{enumerate}
\end{teo}

In \cite{Cabre-SolaMorales, Cabre-Sire:I} it is shown that in the Euclidean case equation \eqref{problem-extension} has a Hamiltonian structure. In addition, Theorems 1.3 and  2.3 in \cite{Cabre-SolaMorales} and \cite{Cabre-Sire:I}, respectively, provide an identity, that implies, among other things, that solutions to  \eqref{problem-extension}  converge as $\gamma \to 1$ to a solution of the local equation. In our case, the Hamiltonian identity does not hold:
There is no conserved quantity along the trajectories, but it dissipates when $t\to+\infty$. However, these results suffice to study the limit when $\gamma\to 1$:

\begin{teo}\label{thm-limit}
Let $\{\gamma_k\}$ be a sequence of real numbers in $(0,1)$ such that $\gamma_k \uparrow 1$ when $k\to\infty$. Let $\{w_k\}$ be a sequence of layer solutions of
$$(-\Delta_{\mathbb H^n})^{\gamma_k} w_k=f(w_k)\quad\text{in }\mathbb H^n,$$
such that $w_k(0)=0$.
Then there exists a function $\overline w$ such that
$$\lim_{k\to\infty} w_{k}=\overline w,$$
where the convergence is $\mathcal C^2$ uniform on compact sets. Moreover,
the function $\overline w$ is the (unique) layer solution of
\be\label{layer1}-\Delta_{\mathbb H^n } \overline w=f(\overline w)\ee
with $\overline w(0)=0$.
\end{teo}

For $\gamma=1$ is possible to construct entire solutions to \eqref{problem-extension} which are not layer solutions. In \cite{PKPW}, the authors constructed an entire solution to \eqref{problem-extension} in $\rr^2$ with nodal set asymptotic to a family of parallel hyperplanes (which satisfy a condition on their separation). A similar result was proved for hyperbolic space in \cite{Mazzeo-Saez}.  We refer to these solutions as {\em  multilayer} solutions.
The construction of multilayer solutions involves gluing techniques that have not been developed for non-local operators. Nonetheless,
following \cite{Gonzalez-Mazzeo-Sire}, we may prove existence of multilayer solutions for $\gamma$ close to 1.  More precisely, we use a perturbation argument and the results in \cite{Mazzeo-Saez} to show:

\begin{teo}\label{thm-multilayer}
Consider a family of non-intersecting totally geodesic hyperplanes $\{\Pi_j\}_{j=1}^N$ in $\hh^n$ which are widely separated. Then, there is $\delta>0$ such that  for every
 $\gamma\in(1-\delta,1]$
there exists a solution $u_\gamma$ that satisfies \eqref{problem-extension} and vanishes asymptotically on $\calH:=\cup_{j=1}^N \Pi_j$.\end{teo}

In the previous theorem by ``widely separated'' we mean that there is a constant $D$ such that if minimal distance among any two planes of the previous configuration is larger than $D$ then the result holds. This is a technical condition that might not be necessary.\\

The structure of the paper is the following: In Section 2 we give a quick introduction about the definition of the fractional Laplacian on hyperbolic space and the extension problem. Then, in Section 3 we give the necessary preliminary results on regularity and maximum principles. Section 4 is a necessary condition for the existence of the layers. The main result on the construction of a layer solution is contained in Section 5. In the next section we give the proof of Theorem \ref{thm-symmetry} on the two dimensional symmetry of solutions. Section 7 contains some auxiliary Hamiltonian estimates that will be used in Section 8, the passage to the limit $\gamma\to 1$. Finally, Section 9 deals with the multilayer construction and the proof of Theorem \ref{thm-multilayer}.

For the rest of the paper, we take $\gamma\in(0,1)$ and thus, $a\in(-1,1)$.

\section{The fractional Laplacian on hyperbolic space}

In this section we summarize the results in \cite{Banica-Gonzalez-Saez}. For further details we refer the reader to that work.
\subsection{The model}\label{model}

Several models of the $n$ dimension hyperbolic space $\hh^n$ have been considered in the literature.
 We may define $\HH^n$ as the upper branch of a hyperboloid in $\rr^{n+1}$ with the metric induced by the Lorentzian metric in $\rr^{n+1}$ given by $ds^2-|dx|^2$. More precisely, we take
\begin{equation*}\begin{split}
\hh^n&=\{(s,x)\in \rr\times \rr^{n}: s^2-|x|^2=1, \; s>0\} \\
&=\{(s,x)\in \rr\times \rr^{n}: (s,x)= (\cosh r, \sinh r \omega), \; r\geq 0, \; \omega \in \ms^{n-1}\},
\end{split}\end{equation*}
with the metric
 $$g_{\mathbb H^n}=dr^2+\sinh^2 r \,d\omega^2,$$ where $d\omega^2$ is the metric on $\ms^{n-1}$.  Notice that the Lorentzian metric induces an internal product that we denote by $[\cdot, \cdot]$. More precisely,
$$[(s_1,x_1), (s_2,x_2)]=s_1 s_2-x_1\cdot x_2.$$
Under these definitions we have that the Laplace-Beltrami operator is given by
$$\Delta_{\hh^n}=\partial_{rr}+(n-1) \frac{\cosh r}{\sinh r} \partial_r+\frac{1}{\sinh^2 r}\Delta_{\ms^{n-1}}$$
and the volume element is $$d\Omega=\sinh^{n-1} r\;dr \; d\omega.$$

Similarly, it is possible to write the Poincar\'e ball model of hyperbolic space. Notice that via stereographic projection one obtains the half-space model:
$$\hh^n=\{(x_1,\ldots,x_n): x_n>0\},$$ with the metric
$$g_{\mathbb{H}^n}=\frac{dx_1^2+\ldots+dx_n^2}{x_n^2}.$$
We choose coordinates $x=(\tilde x,x_n)$, $\tilde x=(x_1,\ldots,x_{n-1})$. The volume element will be denoted by
$$dV_{\mathbb H^n}(x)=(x_n)^{n}\;d\tilde x \;dx_n.$$
Under this parametrization the Laplace-Beltrami operator is given by
\begin{equation}
\label{Laplace-Beltrami}\Delta_{\hh^n} =x_n^2\Delta_x -(n-2)x_n\partial_n .\end{equation}
Here $\Delta_x$ denotes the Euclidean Laplacian in coordinates $x_1,\ldots,x_n$.\\

Let $\Pi$ be a totally geodesic hyperplane in $\mathbb H^n$. Then the hyperbolic metric can be written as the warped product of $\RR \times \HH^{n-1}$, with metric
\be
g = dt^2 + \cosh^2 t\, g_{\HH^{n-1}},
\label{decompH}\ee
where $t$ is the signed distance from $\Pi$, and $z\in\mathbb H^{n-1}$. Note that the plane $\Pi$ divides the boundary $\partial_\infty \mathbb H^n=\mathbb S^{n-1}$ into two hemispheres, denoted by $S_+$ and $S_-$.
In these coordinates,
\be\label{Laplacian1}
\Delta_{\HH^n} = \del_{tt} + (n-1)\tanh t\, \del_t + \sech^2 t\, \Delta_{\HH^{n-1}}.
\ee

\subsection{Fourier transform and the fractional Laplacian on hyperbolic space}

We start by reviewing some basic facts about  Fourier transform on hyperbolic space.
We follow the notation from \cite{Banica-Gonzalez-Saez}, but further references can be found in
 \cite{GGG} and \cite{H}.

Consider the generalized eigenfunctions of the Laplace-Beltrami operator in $\mathbb H^n$:
$$h_{\lambda,\theta}(\Omega)=[\Omega, \Lambda(\theta)]^{i\lambda-\frac{n-1}{2}}, \quad \Omega\in \hh^n,$$
where $\lambda\in \rr$, $\theta \in \ms^{n-1}$ and  $\Lambda(\theta)=(1,\theta)$.
The Fourier transform is defined as
$$\hat{w}(\lambda, \theta)=\int_{\hh^n} w(\Omega)h_{\lambda,\theta}(\Omega)d\Omega,$$
for $\lambda\in\rr$, $\omega\in \ms^{n-1}$. Moreover, the following inversion formula holds:
$$w(\Omega)=\int_{-\infty}^{\infty}\int_{\ms^{n-1}}\bar{h}_{\lambda,\theta}(\Omega)\hat{w}
(\lambda,\theta)
\frac{d\theta \; d\lambda }{|c(\lambda)|^2},$$
where $c(\lambda)$ is the Harish-Chandra coefficient:
$$\frac{1}{|c(\lambda)|^2}=\frac 12\frac{|\Gamma(\frac{n-1}{2})|^2}{|\Gamma(n-1)|^2}\frac{|\Gamma(i\lambda+(\frac{n-1}{2})|^2}{|\Gamma(i\lambda)|^2}.$$
Then one may check that for every $w\in L^2(\hh^n)$,
\begin{align*}\widehat{\Delta_{\hh^n} w}=-\tfrac{(n-1)^2+\lambda^2}{4}\,\hat{w}.
\end{align*}

\begin{definition}  Let $(-\Delta_{\mathbb{H}^n})^\gamma$ be the operator that satisfies
$$\widehat{(-\Delta_{\mathbb{H}^n})^\gamma w}= \left(\lambda^2+\tfrac{(n-1)^2}{4}\right)^\gamma\hat{w}.$$
Equivalently (due to the inversion formula)
\begin{align*}(-\Delta_{\mathbb{H}^n})^\gamma w (\Omega)=&\int_{-\infty}^{\infty}\int_{\hh^n}\left(\tfrac{(n-1)^2+\lambda^2}{4}\right)^\gamma L_\lambda(\Omega, \Omega') w(\Omega')
\frac{d\Omega' \; d\lambda }{|c(\lambda)|^2},
\end{align*}
where we have defined
$$L_\lambda(\Omega, \Omega') =\int_{\ms^{n-1}} \bar{h}_{\lambda,\theta}(\Omega) h_{\lambda,\theta}(\Omega')d\theta.$$
\end{definition}

In \cite{Banica-Gonzalez-Saez} it was shown that the fractional Laplacian may be calculated as the convolution with a radially symmetric, well behaved kernel. In particular:

\begin{teo}\label{thm:singular-integral} It holds that
\begin{align*}(-\Delta_{\mathbb{H}^n})^\gamma w (x)= \hbox{P.V.}\int_{\hh^n}(w(x')-w(x)) K_\gamma(\rho) \,dx',
\end{align*}
where $\rho=d_{\mathbb H^n}(x,x')$ and the kernel is explicitly given by:
\begin{itemize}
\item For $n\geq 3$ odd,
\begin{equation*}K_\gamma(\rho)=C_n\left(\frac{\partial_\rho}{\sinh\rho}\right)^\frac{n-1}{2}
\rho^{-\frac{1}{2}-\gamma}Z_{\frac{1}{2}+\gamma}\left(\tfrac{n-1}{2} \rho\right),\end{equation*}

\item When $n\geq 2$ is even,
$$K_\gamma(\rho)=C'_n
\int_\rho^\infty\frac{\sinh r}{\sqrt{\cosh r-\cosh\rho}}
\left(\frac{\partial_r}{\sinh r}\right)^\frac{n}{2}\left[r^{-\frac{1}{2}-\gamma}Z_{\frac{1}{2}+\gamma}\left(\tfrac{n-1}{2} r\right)\right]\, dr.$$
\end{itemize}
Here $Z_{\frac{1}{2}+\gamma}$ is the solution to the modified Bessel equation given by Lemma 2.2 in \cite{Banica-Gonzalez-Saez}, $C_n,C'_n$ are constants that depend on $n$ and  $\hbox{P.V.}$ denotes the principal value.

Additionally, $K_\gamma(\rho)$ has the asymptotic behavior:
\begin{itemize}
\item[\emph{i.}] As $\rho\to 0$,
 $$K_\gamma(\rho)\sim \frac{1}{\rho^{n+2\gamma}}.$$

\item[\emph{ii.}] As $\rho\to\infty$,
$$K_\gamma(\rho)\sim \rho^{-1-\gamma}e^{-(n-1)\rho}.$$
\end{itemize}

\end{teo}

There are several ways to define the Sobolev spaces on hyperbolic space and more generally on manifolds. We refer to \cite{Triebel:II} and the references therein.
For a given  $n$-dim manifold $M$ with positive injectivity radius and bounded geometry the
Sobolev spaces $W^k_p(M)$ with $k$ integer were first defined as
$$W^k_p(M)=\{f\in L^p(M)\,:\, \nabla_g^l f\in L^p(M),\forall 1\leq l\leq k\},$$
with  norm $\|f\|_{W^k_p(M)}=\Sigma_{l=0}^k\|\nabla_g ^l f\|_{L^p(M)}.$

Next, let $\gamma\in\mathbb R$ and $p\in(1,\infty)$.
 The fractional spaces $H^\gamma_p(M)$ with $\gamma>0$ are
$$H^\gamma_p(M)=\{f\in L^p(M) \,:\,\exists h\in L^p(M), f=(id-\Delta)^{-\gamma/2}h\},\hbox{ with norm } \|f\|_{H^\gamma_p(M)}=\|h\|_{L^p(M)}.$$
A similar definition is given also for $\gamma<0$.  If one considers the hyperbolic space as a symmetric space, the general theory on Fourier multipliers (see, for instance, \cite{Anker:multipliers}, \cite{Tataru:Strichartz-hyperbolic}) gives the following equivalence:
$$H^\gamma_p(\mathbb H^n)=\{f\in L^p(\mathbb H^n) \,: \,\|f\|_{L^p(\mathbb H^n)}+\|(-\Delta_{\mathbb H^n})^{\frac \gamma 2}f\|_{L^p(\mathbb H^n)}<\infty\}.$$
Usual Sobolev embeddings hold. In particular, for $-\infty<\gamma_1\leq \gamma_2<+\infty$ and $0<p<+\infty$,
$$H^{\gamma_1}_{p}(M)\subseteq H^{\gamma_2}_{p}(M).$$
In our notation we will drop the subindex $p$ in the $p=2$ case.

We will also work with weighted Sobolev spaces for a given weight $\omega$, that can be defined by $$W^{1,p}(M, \omega)=\{f\in L^p(M)\,:\, \omega^{\frac{1}{p}}\nabla_g^l f\in L^p(M),\forall 1\leq l\leq k\},$$
with  norm $\|f\|_{W^{k,p}(M, \omega)}=\Sigma_{l=0}^k\|\omega^{\frac{1}{p}}\nabla_g ^l f\|_{L^p(M)}.$

\subsection{The extension problem}

Let $g$ be the product metric in $\hh^n\times \rr_+$ given by $g=g_{\hh^n}+dy^2$. It was shown in \cite{Banica-Gonzalez-Saez} that:

\begin{teo}\label{thm-extension}
Let $\gamma\in(0,1)$. Given $w\in H^\gamma(\mathbb H^n)$, there exists a unique solution of the extension problem
\begin{equation}
\begin{cases}
\divergence_g(y^a\nabla_g u)(x,y) = 0 &\quad\hbox{for } (x,y)\in \hh^n\times \rr_+,\\
u(x,0)=w(x) &\quad \hbox{for } x\in \hh^n,
\end{cases}\label{extension}\end{equation}
Moreover,
\begin{equation*}(-\Delta_{\mathbb{H}^n})^\gamma w=-d_{\gamma}\lim_{y\to 0}y^a \partial_y u,\end{equation*}
for a constant
\begin{equation}\label{constant-d}d_{\gamma}=2^{2\gamma-1}\frac{\Gamma(\gamma)}{\Gamma(1-\gamma)}.\end{equation}
The solution of \eqref{extension} is given explicitly by the convolution
\begin{equation}\label{Poisson-extension}
u(x,y)=\int_{\mathbb H^n} \mathcal P^\gamma_y(\rho)w(x')\,dV_{\mathbb H^n}(x').
\end{equation}
with $\rho=d_{\hh^n}(x,x')$ the hyperbolic distance between $x$ and $x'$, and the Poisson kernel is  written as
$$\mathcal P^\gamma_y(\rho)=\int_{-\infty}^\infty k_{\lambda}(\rho)\,\varphi_\gamma\lp\left(\lambda^2+\tfrac{(n-1)^2}{4}\right)^{1/2}y\rp\,d\lambda,$$
where $k_\lambda(\rho)$ is defined by
\begin{equation*} k_\lambda(\rho)=
\left(\frac{\partial_\rho}{\sinh\rho}\right)^\frac{n-1}{2}(\cos\lambda\rho)\end{equation*}
for $n\geq 3$ odd, and for $n\geq 2$ even,
\begin{equation*}
k_\lambda(\rho)=
\int_\rho^\infty\frac{\sinh r}{\sqrt{\cosh r-\cosh\rho}}
\left(\frac{\partial_r}{\sinh r}\right)^\frac{n}{2}(\cos\lambda r)\, dr.
\end{equation*}
We moreover have the energy equality:
\begin{equation*}
\int_{\hh^n\times \rr_+}  y^a|\nabla_g u|^2 \, dV_{\mathbb H^n}(x)\, dy=d_\gamma^{-1} \int_{\hh^n}|(-\Delta_{\hh^n})^\frac\gamma2 w(x)|^2\,dV_{\mathbb H^n}(x),
\end{equation*}
\end{teo}

\begin{teo}[Trace Sobolev embedding]\label{thm:trace-Sobolev}
For every $u\in W^{1,2}(\mathbb H^n\times\mathbb R_+, y^a)$, we have that
$$\|\nabla u\|^2_{L^2(\mathbb H^n\times\mathbb R_+, y^a)}\geq d_\gamma^{-1}\|u(\cdot,0)\|^2_{H^\gamma(\mathbb H^n)}$$
for the constant given in \eqref{constant-d}, and with equality if and only if $u$ is the Poisson extension \eqref{Poisson-extension} of some function in $H^\gamma(\mathbb H^n)$.
\end{teo}

We remark here that our weight $\omega:=y^a$ on $\mathbb H^n\times\mathbb R_+$ is of type $\mathcal A_2$ in the Muckenhoupt sense.

\section{Preliminary results} \label{preliminary}

We define for $x_0 \in \hh^n$
\begin{equation*}\begin{split}B^+_R(x_0)=&\{(x,y)\in  \hh^n\times \rr_+: d_{\hh^n}^2(x,x_0)+|y|^2\leq R^2\},\\
\Gamma^0_R(x_0)=&\{(x,0)\in  \hh^n\times \rr_+:  d_{\hh^n}^2(x,x_0)\leq R^2\},\\
\Gamma^+_R(x_0)=&\{(x,y)\in  \hh^n\times \rr_+:  d_{\hh^n}^2(x,x_0)+|y|^2=R^2\}.\end{split}\ees
where $d_{\hh^n}$ denotes the distance in $\hh^n$.

We first describe a concept of weak solutions for problem \eqref{problem-extension}:

\begin{definition}
{\rm
Given $R>0$ and a function $h \in L^1(\Gamma^0_R)$, we say that $u$ is a
weak solution of
\begin{equation}\label{temp}
\begin{cases}
\divergence_g (y^a \nabla_g u)=0&\text{in }B_R^+,\\
-y^a u_y|_{y=0}=h&\text{on }\Gamma_R^0,
\end{cases}
\end{equation}
if
$$y^a |\nabla_g u|^2 \in L^1(B_R^+)$$
and
\begin{equation}\label{weak11}
\int_{B_R^+} y^{a} \nabla_g u \cdot \nabla_g \xi \,dV_{\mathbb H^n}(x)dy-\int_{\Gamma^0_R} h
\xi \,dV_{\mathbb H^n}(x)=0
\end{equation}
for all $\xi \in \mathcal C^1(\overline{B_R^+})$ such that $\xi\equiv 0$ on $\Gamma^+_R$.}
\end{definition}

\subsection{Regularity}\label{subsection-regularity}

The following results may be proved exactly as in \cite{Cabre-Sire:I}. Indeed, the kernel representation of the fractional Laplacian on hyperbolic space from Theorem \ref{thm:singular-integral}  allows to get regularity estimates as in \cite{Silvestre:regularity-obstacle} (see also \cite{Banica-Gonzalez-Saez}) and the structure of the equation in local coordinates allows to use the results in \cite{FKS} for degenerate elliptic equations with $\mathcal A_2$ weight.

\begin{lemma}\label{regNL}
Let $f$ be a $\mathcal C^{1,\alpha}(\RR)$ function with $\alpha >\max\{0,1-2\gamma\}$.
Then, any bounded solution of
$$(-\Delta_{\mathbb H^n})^\gamma w =f(w)\quad\mbox{ in }\HH^n$$
is $\mathcal C^{2,\beta}(\HH^n)$ for some $0 <\beta < 1$ depending only on $\alpha$ and $\gamma$.
In addition, the Poisson extension $u$ of $w$ as given in \eqref{Poisson-extension} satisfies
$$\|u\|_{\mathcal C^\beta(\overline{\mathbb H^n\times \mathbb R_+ })}+\|\nabla_{\mathbb H^n} u\|_{\mathcal C^\beta(\overline{\mathbb H^n\times \mathbb R_+ })}+\|D_{\mathbb H^n}^2 u\|_{\mathcal C^\beta(\overline{\mathbb H^n\times \mathbb R_+ })}\leq C,$$
for some constant $C$ depending only on $n$, $\gamma$, $\|f\|_{\mathcal C^{1,\alpha}}$,
and $\|w\|_{L^{\infty}(\HH^n)}$.

Moreover, if $\gamma_0>1/2$, then the these estimates are valid for all $\gamma\in(\gamma_0,1)$ where $\beta$ and $C$ may be taken depending only on $\gamma_0$ and uniform in $\gamma$.
\end{lemma}

\begin{lemma}\label{regularity1} Let $R>0$. Let $h \in \mathcal C^\alpha (\Gamma^0_{2R})$ for some $\alpha \in (0,1)$ and $u \in
L^\infty(B^+_{2R}) \cap W^{1,2}(B^+_{2R},y^a)$ be a weak solution of
\begin{equation*} \label{problemBR} \begin{cases}
\divergence_g (y^a \nabla_g u)=0&\text{ in } B^+_{2R},\\
-y^a\partial_y u|_{y=0} =h&\text{ on }
\Gamma^0_{2R}. \end{cases} \end{equation*}
Then, there exists  $\beta \in (0,1)$ depending only on $n$, $a$, and $\alpha$, such that $u \in
\mathcal C^\beta(\overline{B_R^+})$ and $y^a u_y \in \mathcal C^\beta(\overline{B_R^+})$.

Furthermore, there exist constants $C^1_R$ and $C^2_R$ depending only on $n$, $a$, $R$, $\|u\|_{L^\infty(B_{2R}^+)}$
and also on $\|h \|_{L^\infty(\Gamma^0_{2R})}$ (for $C^1_R$) and $\|h \|_{\mathcal C^\sigma(\Gamma^0_{2R})}$ (for
$C^2_R)$, such that \begin{equation*}
 \|u\|_{\mathcal C^\beta(\overline{B_R^+})} \leq C^1_R
\end{equation*} and \begin{equation*}
 \|y^a u_y\|_{\mathcal C^\beta(\overline{B_R^+})} \leq C^2_R.
\end{equation*}
\end{lemma}

\subsection{Maximum principles}\label{subsection:maximum}

\begin{rema}\label{MPweak}
{\rm
The (weak) maximum principle holds for weak solutions of \eqref{temp}.
More generally, if $u$ solves
\begin{equation*}
\begin{cases}
-\divergence_g (y^a \nabla_g u)\geq 0&\text{in $B_R^+$},\\
-y^a u_y\geq 0 &\text{on $\Gamma_R^0$},\\
u\geq 0 &\text{on $\Gamma_R^+$,}
\end{cases}
\end{equation*}
in the weak sense, then $u\geq 0$ in $B_R^+$. This is proved simply inserting the negative part $u^-$ of $u$ in the weak formulation \eqref{weak11}.

In addition, one has
the strong maximum principle: either $u\equiv 0$ or
$u>0$ in $ B_R^+\cup \Gamma_R^0$. That $u$ cannot vanish at an
interior point follows from the classical strong maximum principle
for strictly elliptic operators. That $u$ cannot vanish at a
point in $\Gamma_R^0$ follows from the Hopf principle
that we establish below (see Lemma~\ref{lemma-hopf}) or
by the strong maximum principle of \cite{FKS}.}
\end{rema}

\begin{lemma}\label{lemma-hopf}
Consider the cylinder
$C_{R,1}= \Gamma_R^0 \times (0,1) \subset \HH^{n}\times \RR^+ $ where $\Gamma_R^0$
is the ball of center $o$ and radius $R$ in $\HH^n$. Let $u \in \mathcal C(\overline{C_{R,1}})
\cap W^{1,2}(C_{R,1},y^a)$ satisfy
\begin{equation*}
\begin{cases}
 -\divergence_g (y^a \nabla_g u) \leq 0&\text{ in } C_{R,1}, \\
u> 0&\text{ in } C_{R,1}, \\
u(o,0)=0.&
\end{cases}
\end{equation*}
Then,
$$\limsup_{y \rightarrow 0^+ }-y^a \frac{u(o,y)}{y}<0.$$
In addition, if $y^a u_y \in \mathcal C(\overline{C_{R,1}})$, then
$$-\lim_{y\to 0} y^a  \partial_y u (o,0) <0. $$
\end{lemma}

We remark here that this version of the Hopf's lemma will work for any product manifold $M^n\times \mathbb R_+$ with the product metric $g=g_{\mathbb H^n}+dy^2$, and the proof is exactly the same as in the flat case (see \cite{Cabre-Sire:I}). For more general manifolds with boundary we refer to \cite{Chang-Gonzalez,Gonzalez-Qing}.

\begin{lemma}\label{lemma-max}
Fix $\varepsilon >0$. Let $d$ be a H\"older continuous function in
$\Gamma^0_\varepsilon$ and $u \in L^\infty(B^+_\varepsilon) \cap W^{1,2}(B^+_\varepsilon,y^a)$
be a weak solution of
\begin{equation*}
\begin{cases}
\divergence_g (y^a \nabla_g u)=0 &\text{ in } B^+_\varepsilon, \\
u \geq  0&\text{ in }  B^+_\varepsilon,\\
- y^a \partial_y u+d(x) u|_{y=0}=0&\text{ on } \Gamma^0_\varepsilon.
\end{cases}
\end{equation*}
Then, $u >0$ in $B^+_\varepsilon \cup \Gamma^0_\varepsilon$ unless $u \equiv 0$ in $B^+_\varepsilon. $
\end{lemma}

\begin{lemma}
\label{lemma-max1}
Let $u\in (\mathcal C\cap L^\infty)(\overline{\HH^{n}\times \RR^+})$ with $y^a u_y \in \mathcal C(\overline{\HH^{n}\times \RR^+})$
satisfy
\begin{equation*} \label{maxpr}
\begin{cases}
 \divergence_g (y^a \nabla_g u)= 0&\text{ in } \HH^{n}\times \R^+,\\
 - y^a\partial_y u+d(x)u|_{y=0}\ge 0&\text{ on } \HH^{n},
\end{cases}
\end{equation*}
where $d$ is a bounded function,
and also
\begin{equation*}
\label{lim0}
u(x,0)\rightarrow 0 \qquad \mbox{as } x\to\partial_\infty {\hh^n}.
\end{equation*}
Assume that there exists a nonempty set $\Sigma\subset\HH^n$ such that
$u(x,0)> 0$ for $x\in \Sigma$, and $d(x)\ge 0$ for $x\not\in \Sigma$.

Then, $u>0$ in $\overline{\HH^{n}\times \RR^+}$.
\end{lemma}

\section{A necessary condition}

The extension problem \equ{extension} is variational. Indeed, the associated energy in a bounded Lipschitz domain $\Omega\subset \hh^n\times \rr_+$ is given by
\begin{equation} E_\Omega(u)=\int_{\Omega}y^a\frac{|\nabla_g u|^2}{2}(x,y)\,dV_{\hh^n}(x) dy+\int_{\partial\Omega\cap \{y=0\}} F(u)(x,0)\,dV_{\hh^n}(x).\label{energy1}\end{equation}
We say that $u$ is a \emph{local minimizer} of \eqref{problem-extension} with respect to relative perturbations in $[-1,1]$ if
$$E_{B_R^+}(u)\leq E_{B_R^+}(u+\psi)$$
for every $R>0$ and for every $\mathcal C^1$ function $\psi$ on $\overline{\mathbb H^n\times \mathbb R^+}$ with compact support in $B_R^+\cup \Gamma_R^0$ and such that $-1\leq u+\psi\leq 1$ in $B_R^+$.

We start with a necessary condition that, for the Euclidean case, is contained in Proposition~5.2 of \cite{Cabre-Sire:II}. Let $\Pi$ be a totally geodesic hyperplane in $\mathbb H^n$, and use coordinates $t>0$, $z\in\mathbb H^{n-1}$, where $t$ is the signed distance from $\Pi$.

\begin{prop}\label{sameheight}
 Let $u$ be a solution of \equ{problem-extension} such that $|u|<1$, and \begin{equation*}
\label{limitsL2} \lim_{t\to\pm\infty}u(t,z,0)=L^\pm  \quad\text{for every } z\in \hh^{n-1},
\end{equation*} for some constants $L^-$ and $L^+$ $($that could be equal$)$. Assume that $u$ is a local minimizer
relative to perturbations in $[-1,1]$. Then, \begin{equation*}\label{Gge} F\ge F(L^-) = F(L^+)\quad \text{ in } [-1,1].
\end{equation*}
\end{prop}

\begin{proof}
As in \cite{Cabre-Sire:II}, it suffices to prove that  $F\ge F(L^-)$ and $F\ge F(L^+)$.
Since the proofs of both inequalities are analogous, it is enough to show one of them. We will establish the second one by a contradiction argument. Moreover, since the solution is independent of considering translations of $F$, we may assume that there exists a point $s$ such that
$$
F(s)=0<F(L^+) \quad \text{ for some } s\in [-1,1].$$
 Since $F(L^+) > 0$, we
have that $$ F(\tau) \geq \varepsilon > 0 \quad\text{ for $\tau$ in a neighborhood in $[-1,1]$ of } L^+ $$ for
some $\varepsilon>0$.

Consider the points $(b,0,0)$ on $\partial(\hh^{n}\times \rr_+)$ (that is points,  where $t=b, z=0,y=0$). Since for $T>0$, $$
E_{B^+_T(b,0,0)}(u) \geq \int_{\Gamma^0_T(b,0)} F(u(x,0))\, dV_{\hh^n}(x) $$ and, since
in $\hh^n$ the volume of the ball of radius $T$ is given by $\omega_n\int_0^T\cosh^{n-1}t dt$, for large $T$ we have $\text{Vol}(\Gamma^0_T(b,0))\sim
c(n) e^{(n-1)T}$ and
 $u(x,0) \underset{x_1 \to +
\infty}{\longrightarrow} L^+$, we deduce
\begin{equation} \label{lowbou} \varliminf_{b\to +\infty} E_{B^+_T(b,0,0)}(u)
\geq c(n)\, \varepsilon e^{(n-1)T} \qquad \text {for all }  T > 1, \end{equation}
where the constant $c(n)$ depends only on $n$.

The lower bound \eqref{lowbou} will give a contradiction with the upper bound \eqref{upbou} for the energy of $u$, that is obtained using the local minimality of $u$.
For every $T>1$,  $b\in \rr$ and  $\eta \in (0,1)$, we may define a smooth function $\xi_{T,b}$  in $\hh^n\times \RR_+$ that satisfies $0\leq \xi_{T,b}\leq 1$, $$ \xi_{T,b}= \begin{cases}
\displaystyle 1 & \text{ in} \,\,\,B^+_{(1-\eta)T}(b,0,0), \\ 0 & \text{ on} \,\,\,(\hh^n\times\RR_+) \backslash B_T^+(b,0,0) , \end{cases}
$$ and $|\nabla_{\hh^n\times \rr_+}\xi_{T,b}|\leq C(n)(\eta T)^{-1}e^{-(1-\eta)T}$.  Since $$ (1-\xi_{T,b})u+\xi_{T,b}s=u+\xi_{T,b}(s-u) $$ takes values
in $[-1,1]$ and agrees with $u$ on $\Gamma^+_T(b,0,0)$, we have that $$ E_{B^+_T(b,0,0)}(u) \leq
E_{B^+_T(b,0,0)}(u+\xi_{T,b}(s-u)).  $$

Using that  $F(s)=0$ we have that the potential energy is only nonzero in $B^+_{T}\setminus
B^+_{(1-\eta)T}$. Since  $\text{Vol}((B^+_{T}(b,0,0)\setminus
B^+_{(1-\eta)T}(b,0,0))\cap \{y=0\})=\text{Vol}(\Gamma^+_{T}(b,0,0)\setminus
\Gamma^+_{(1-\eta)T}(b,0,0))\leq C(n)\eta T e^{(n-1)T}$, we have that
\begin{equation}\label{energy10}\int_{\Gamma^+_{T}(b,0,0)} F(u+\xi_{T,b}(s-u))dV_{\hh^n} \leq C(n)\eta T e^{(n-1)T}.\end{equation}
On the other hand, applying the gradient estimate from Lemma \ref{lemma-gradient}, we deduce that \begin{eqnarray} & & \hspace{-2cm}\varlimsup_{b\to +\infty} \int_{B^+_T(b,0,0)} y^a
|\nabla_{\hh^n\times \rr}\{u+\xi_{T,b}(s-u)\}|^2  \leq 2\int_{B_T^+} y^a |\nabla_{\hh^n\times \rr}\xi_{T,b}|^2 dV_{\hh^n}dy \notag\\ &\leq&  \frac{C(n)}{\eta^2 T^2} e^{(n-1)T-2(1-\eta)T} \int_0^T
y^a\,dy =  C(n)  e^{(n-1)T-2(1-\eta)T} \frac{T^{1+a-2}}{\eta^2}.\label{energy11} \end{eqnarray} Putting together the bounds
for Dirichlet and potential energies \eqref{energy11}-\eqref{energy10}, we conclude that
\begin{equation}\label{upbou}
\begin{split} \varlimsup_{b\to +\infty} E_{B^+_T(b,0,0)}(u)
&\leq  \varlimsup_{b\to +\infty}E_{B^+_T(b,0,0)}(u+\xi_{T,b}(s-u))\\
 &\leq  C\{ \eta T e^{(n-1)T}
 + \eta^{-2}e^{(n-1)T-2(1-\eta)T} T^{1+a-2}\},
\end{split}\end{equation} for some constant $C>0$ depending only on $n$, $a$, and $F$.

Now, choosing a suitable $\eta=\eta(T)$  gives a contradiction between \eqref{lowbou} and \eqref{upbou} for $a<1$ and $T$ large.
 \end{proof}

\begin{lemma}\label{lemma-gradient}
Let $u$ be a bounded solution of \eqref{problem-extension} such that
$$\lim_{t\to\pm \infty} u(t,z,0)=L^{\pm}\quad \text{for every}\quad z\in\mathbb H^{n-1}$$
for some constants $L^+$ and $L^-$ (that could be equal). Then
$$ \|\nabla_{\hh^n} u\|_{L^\infty(B^+_T(x,0))} \rightarrow 0\,\,\mbox{as $t
  \rightarrow \pm \infty$},
$$
\end{lemma}
\begin{proof}
We follow the proof of Lemma 4.8 in \cite{Cabre-Sire:I}. Let $\Pi=\{t=0\}$ and consider a sequence of isometries $i_n$ that satisfies $i_n(\Pi)$ converges to a point on the hyperbolic boundary. Let $u_n(z,y)=u(i_n(z),y)$ for $z\in \HH^n$ and $y\in \rr^+$. These functions satisfy \eqref{problem-extension}, hence, using the uniform H\"older estimates from Lemma \ref{regularity1}, we have that they converge locally uniformly. From the initial hypothesis on $u$, we have that the limit has to be identically a constant.
Finally, the uniform $\mathcal C^\beta$ estimate for $|\nabla_{\mathbb H^n} u|$ from Lemma \ref{regNL} finishes the proof.

\end{proof}


\section{The one-dimensional solution}

In the section we provide the proof of Theorem \ref{thm-one-dimensional}. We look for a one-dimensional solution $w$ to problem \eqref{initial-problem} that depends only on the signed distance to the fixed  totally geodesic hyperplane $\Pi$. In the light of Theorem \ref{thm-extension}, it is equivalent to find a function $u(t,y)$ satisfying \eqref{problem-extension}. In particular, $u$ is a solution to
\begin{equation*}
\begin{cases}
H_au=0, &\quad\text{for }(t,y)\in\mathbb R\times \mathbb R_+,\\
-y^a \partial_y u|_{y=0}=f(u),&\quad\text{for }t\in\mathbb R,
\end{cases}\end{equation*}
where the operator $H_a$ is given by
$$H_au=\del_{yy}+\frac{a}{y} \del_y  +\del_{tt} + (n-1)\tanh t\, \del_t\,.$$
We will follow the arguments in \cite{Cabre-Sire:II} for the Euclidean case with appropriate modifications.  The main difficulty in our setting is the lack of translation invariance.
However, the structure of the proof remains unchanged.

Notice that one dimensional solutions have infinite energy with respect to \eqref{energy1} (when the domain is unbounded respect to the variable on $\Pi$). Nonetheless, the reduced problem is variational. Indeed, we may consider the one-dimensional energy functional in a domain $\Omega\subset \mathbb R\times\mathbb R_+$, that is written as
\be  E_\Omega(u)=\int_{\Omega} \left[\frac{ |\del_t u|^2+  |\del_y u|^2}{2}\right]y^a (\cosh t)^{n-1}dt dy+\int_{\partial\Omega\cap \{y=0\}} F(u)(t,0)
(\cosh t)^{n-1} dt.\label{defE}\ee
where $-F'=f$.

In what follows, we will consider only this reduced energy (which, in order to simplify notation,  we denote also as $E_\Omega$). In addition, note that the proof of  Proposition \ref{sameheight} and Lemma \ref{lemma-gradient} can be carried in the same fashion for one dimensional solutions to \eqref{1dimsol} with finite one-dimensional energy.

We also remark that the first order quantities in \eqref{defE} correspond to the gradient in $\rr \times \rr_+$. Hence we denote $|\del_t u|^2+  |\del_y u|^2=|\nabla_{\rr \times \rr_+} u|^2$.

\subsection{Local solutions}

In the following, we will be using the coordinates $(t,z)\in \rr \times \hh^{n-1}$ given by
\equ{decompH}.
We will work with a domain $\Omega_{T,R}\subset \mathbb R\times \mathbb R^+$ given by
$$\Omega_{T,R}:=\{(t,y): -T<t<T, y\in(0,R)\}.$$
Define the partial boundary
$$\partial^+\Omega_{T,R}:=\overline{\partial\Omega_{T,R} \cap \{y>0\}}.$$
Fix a weight $\omega_a= y^a(\cosh t)^{n-1}$. Given $v\in \mathcal C^\beta(\overline\Omega_{T,R})\cap W^{1,2}(\Omega_{T,R},\omega_a)$ satisfying $|v|\leq 1$, consider the class
\begin{equation*}
\mathcal A_v:=\{u\in W^{1,2}(\Omega_{T,R},\omega_a) : |u|\leq 1 \mbox{ a.e. in }\Omega_{T,R}, u=v\mbox{ on }\partial^+ \Omega_{T,R} \}
\end{equation*}

\begin{lemma} \label{compacts}
Let $v\in \mathcal C^\beta(\overline\Omega_{T,R})\cap W^{1,2}(\Omega_{T,R},\omega_a) $ be a given function satisfying $|v|\leq 1$, where $\beta\in(0,1)$. Assume that
\begin{equation*}\label{subsuper} f(1)\le 0 \le f(-1). \end{equation*}
Then the functional
$E_{ \Omega_{T,R} }(u)$ defined by \eqref{defE}
admits an absolute minimizer $u_{T,R}$ in the class $\mathcal A_v$ that is a weak solution to the problem
\begin{equation*}
\begin{cases}
H_a u=0 &\mbox{ in }\Omega_{T,R},\\
-y^a  \partial_y u|_{y=0} =f(u) &\hbox{ for }t\in [-T,T],\\
u= v &\mbox{ on }   \partial^+\Omega_{T,R}.
\end{cases}
\end{equation*}
Moreover, $u_{T,R}$ is stable in the sense that
$$\int_{-T}^T
\int_{0}^R \left[\frac{ |\del_t \xi|^2+  |\del_y \xi|^2}{2}\right]y^a (\cosh t)^{n-1}dt dy+\int_{-T}^T f'(u_{T,R})(t,0)
(\cosh t)^{n-1}\xi^2 dt\geq 0,$$
 for every $\xi \in W^{1,2}(\Omega_{T,R},\omega_a)$ such that $\xi\equiv 0$ on $\partial^+\Omega_{T,R}$ in the weak sense.
\end{lemma}

\begin{proof}
The proof is the same as Lemma 4.1 in \cite{Cabre-Sire:II}.
It is useful to consider the following continuous extension $\tilde f$ of $f$ outside $[-1,1]$:
\begin{equation*}
\hat f (s)= \begin{cases} f(-1)&\text{ if
} s\le -1,\\ f(s) &\text{ if } -1\le s \le 1,\\ f(1) &\text{ if } 1\le s.
\end{cases} \end{equation*}
Let $$ \hat F(s)=-\int_0^s \hat f, $$
and consider the new functional
$$ \hat E[u]:= \int_{-T}^T
\int_{0}^R \left[ \frac{|\del_t u|^2+  |\del_y u|^2}{2}\right]y^a (\cosh t)^{n-1}dt dy+\int_{-T}^T \hat F(u)(t,0)
(\cosh t)^{n-1} dt,$$
in the class
$$\mathcal A'_v:=\{u\in W^{1,2}(\Omega_{T,R},\omega) :  u=v\mbox{ on }\partial^+ \Omega_{T,R} \}.$$
Note that $\tilde F=F$ in $[-1,1]$, up to an additive constant. Therefore, any minimizer $u$ of $\hat E$ in
$\mathcal A'_v$ such that $-1\le u \le 1$ is also a minimizer of the original functional $E$ in $\mathcal A_v$.

To show that  $\hat E$ admits a minimizer in $\mathcal A'_v$, we use a standard compactness argument.
Indeed, let $u\in\mathcal A'_v$. Since $u-v\equiv 0$ on $\partial^+\Omega_{T,R}$, we can extend $u-v$ to be identically
$0$ in $\mathbb R\times\rr_+\setminus \Omega_{T,R}$, and we have $u-v\in W^{1,2} (\mathbb R \times \mathbb R^+,\omega_a)$. Traces of functions in this space belong to $H^\gamma(\mathbb R,(\cosh t)^{n-1})$, and  we have the (compact) embedding $H^\gamma(D, (\cosh t)^{n-1}) \hookrightarrow L^{p}(D,(\cosh t)^{n-1}),$
for any compact domain $D\subset \mathbb R$.

Stability follows by taking a second order variation of the functional $\hat E$ in the space $\mathcal A'_v$. Since $\hat f$ is a continuous function and $\hat E$ is a $\mathcal C^1$ functional in $\mathcal A'_v$.  Therefore, it only remains to show that the minimizer $w$ satisfies $$ -1\le w \le 1 \quad\text{a.e. in } \Omega, $$
but this does not present any further difficulty than in \cite{Cabre-Sire:II}.
\end{proof}

\subsection{Existence: the limit $T\to\infty$}

In the following proposition we are going to construct a layer solution by passing to the limit $T\to\infty$ with the local solutions constructed in the previous subsection. We denote the reduced ball
\begin{align*}S_T^+= &\{(t,y)\;:\;
t^2+y^2\leq T^2, \; y>0\}.\\
\end{align*}

\begin{prop}\label{keyexist} Assume that  $$ F'(-1) = F'(1) = 0 \quad\text{ and }\quad F> F(-1)=F(1)\
\text{ in } (-1,1). $$ Then, for every $T>0$, there exists a function $u^T\in \mathcal C^\beta(\overline{S_T^+})$ for some
$\beta \in (0,1)$ independent of $T$, such that
\begin{align*}
 &-1<u^T< 1 \quad\text{in }\overline{S_T^+}, \\
& u^T(t_0,0) = 0 \quad\hbox{for some }t_0\in \rr, \\
 &\partial_t u^T \ge 0\quad \text{in } S_T^+ ,
 \end{align*}
 and $u^T$ is a minimizer of the energy in $S_T^+$, in the sense that
$$ \ E_{S_T^+}(u^T)\le \ E_{S_T^+}(u^T+\psi) $$ for every $\psi\in \mathcal C^1(\overline{S_T^+})$ with compact support in
$S_T^+\cup\Gamma_T^0$ and such that $-1\le u^T+\psi\le 1$ in $S_T^+$.

Moreover, as a consequence of the previous statements, we will deduce the existence of a subsequence of $\{u^T\}$ which converges in
$\mathcal C^\beta_{\rm loc}(\overline{\R^2_+})$ to a one-dimensional solution of \equ{1dimsol}.
\end{prop}

\begin{proof}
This is the analogous to Lemma 7.1 in  \cite{Cabre-Sire:II}. However, there is an important difference: the choice of the comparison function \eqref{defv}. For $T >1$, let $$ Q_T^+ = (-T,T) \times (0,T^{1/8}),\quad \partial^0 Q_T^+=(-T,T)\times \{0\},\quad \partial^+ Q_T^+=\overline{\partial Q_T^+\cap \{y>0\}}.$$
Consider the function \be  v(t,y) =v(t)=
\tanh(\mu t) \quad \text{ for } (t,y)\in \overline{Q_T^+}, \quad \mu>n-1. \label{defv}\ee
Note that $-1\le v \le 1$ in $Q_T^+$.

Let $u^T$ be the absolute minimizer of Lemma \ref{compacts} for $R=T^
{\frac{1}{8}}$ for $v$ given by \equ{defv}. This function solves the equation
\begin{equation} \begin{cases}
H_a u^T=0&\text{ in } Q_T^+,\label{equation100}\\
-y^a\partial_y u^T =f(u^T)&\text{ on }\partial^0 Q_T^+,\\
|u^T| < 1 \quad &\text{ in } \overline{Q_T^+},\\
u^T=v &\text{ on }\partial^+ Q_T^+.
\end{cases} \end{equation}
The function $u^T$ is H\"older continuous by Lemma~\ref{regularity1}. We will show:
\begin{itemize}
\item \emph{Claim 1:}
\be \label{enerur1}  \quad E_{Q_T^+}(u^T) \leq C T^{1/4} \hbox{ for some constant $C$ independent of $T$. } \end{equation}

\item \emph{Claim 2:}

 \begin{equation}\label{sets1}  \quad|\{u^T(\cdot,0) > 1/2\}| \geq T^{1/4} \text{
and } |\{u^T(\cdot,0) <-1/2\}| \geq T^{1/4}. \end{equation}

\item \emph{Claim 3:}
 \begin{equation*} \label{lastmon}  \partial_t u^T \ge 0 \quad
\text{ in } Q_T^+. \end{equation*}

\item From the previous claims we conclude the existence of a limit satisfying the conditions of the proposition.

 \end{itemize}

\bigskip

\noindent {\it \large Step 1.}

\medskip

We consider
$F-F(-1)=F-F(1)$ as boundary energy potential.

\smallskip
Since $\ E_{Q_T^+}(u^T) \leq \ E_{Q_T^+} (v)$, we simply need to bound the energy of $v$. We have $$ \abs{\nabla_{\rr \times \rr_+} v} =
\abs{\partial_t v} = \mu \sech^2(\mu t), $$ and hence
\begin{equation}\label{energy20} \begin{split}
 \int_{Q_T^+}
\abs{\nabla_{\rr \times \rr_+} v}^2 y^a(\cosh t)^{n-1}\,dtdy=& \mu^2T^{\frac{1+a}{8}} \int_{-T}^T\sech^4(\mu t) (\cosh t)^{n-1}dt\\
\leq & C T^{\frac{1}{4}} (1-e^{(n-1-4\mu )T}). \end{split}
\end{equation}
Using that $F\in \mathcal C^{2,\gamma}$, $F'(-1) = F'(1) = 0$ and $F(-1) = F(1)$, we have that $$ |F(s) -F(1)| \leq C|s-1| \quad\text {for all }s\in [-1,1], $$ for the constant $C=\sup_{s\in[-1,1]} |F'(s)|>0$.
Therefore, \begin{equation*}  |F(v(t,0))-F(1)| \le C\left|v(t,0)-1\right|\leq Ce^{-2\mu t}
\hbox{ for  } t\ge 0.
\end{equation*}
Similarly,
 \begin{equation*}  |F(v(t,0))-F(-1)| \le C\left|v(t,0)+1\right|\leq C e^{2\mu t}
\hbox{ for  } t\le 0.
\end{equation*}
We conclude that $$ \int_{-T}^T \{F(v(t,0))-F(1)\}\, (\cosh t)^{n-1} dt \leq CT(1-e^{(-2\mu+n-1) T}).$$
This, together with the above bound for the Dirichlet energy \eqref{energy20}, gives an upper bound for $E_{Q_T^+}(v)$, which proves \eqref{enerur1}.

\bigskip

\noindent  {\it \large Step 2.}

\medskip
Here we prove \eqref{sets1} for $T$ large enough.

Since $u^T \equiv v$ on $\{y = T^{1/8}\}$ and $\int_{-T}^T v(t)\, dt = 0$, we have \begin{equation}\label{formula10}
\begin{split}
\int_{-T}^T
u^T(t,0)\, dt  &= \int_{-T}^T u^T (t,0) \, dt - \int_{-T}^T u^T(t,T^{1/8})\, dt \\
&= -\int_{Q_T^+} \partial_y u^T\,dtdy. \end{split}\end{equation}
The previous energy bound \eqref{enerur1} and the hypothesis that $F-F(1) \geq 0$ give that the Dirichlet energy alone also
satisfies the bound in \eqref{enerur1}.
Writing $$|\partial_y u^T|= y^{-a/2}(\cosh t)^{\frac{1-n}{2}} y^{a/2} (\cosh t)^{\frac{n-1}{2}}|\partial_y u^T|$$ and from \eqref{formula10},
using the Cauchy-Schwarz inequality, we have
 \begin{equation*} \label{intur} \begin{split} \Big|\int_{-T}^T
u^T(t,0) dt\Big| & = \int_{Q_T^+} |\partial_y u^T|\,dt
   \leq \bigg\{ \int_{Q_T^+} y^{-a} (\cosh t)^{1-n}\,dtdy \int_{Q_T^+} y^a (\cosh t)^{n-1} \abs{\nabla_{\rr \times \rr_+}
    u^T}^2 \,dtdy\bigg\}^{\frac{1}{2}} \\
  & \leq C \Big\{ T^{(1-a)/8}T^{1/4}\Big\}^{1/2}  \leq CT^{1/4},
\end{split} \end{equation*}
where we have used Claim 1 and the fact that $0 < 1-a < 2$.
In particular,
\begin{equation} \label{secintur} \Big| \int_{(-T,T) \cap
\{\abs{u^T(\cdot,0)} > 1/2\}} u^T(t,0)\, dt \Big|  \leq  CT^{1/4}. \end{equation}
On the other
hand, $F(s)-F(1) \geq \varepsilon > 0$ if $s \in [-1/2, 1/2]$, for some $\varepsilon > 0$ independent of $R$, and $F-F(1) \geq 0$ in $(-1,1)$. Moreover, by  \eqref{enerur1} we have $\int_{-T}^T \{ F(u^T(t,0))-F(1)\} \,(\cosh t)^{n-1} dt\le CT^{1/4}$. We deduce $$ \varepsilon \big|\{\abs{u^T(\cdot,0)} \leq 1/2\}\big| \leq \int_{-T}^T \{
F(u^T(x,0))-F(1) \}\, (\cosh t)^{n-1}dt \leq CT^{1/4}, $$ and therefore
$$\big|\{\abs{u^T(\cdot,0)} \leq 1/2\}\big| \leq  CT^{1/4}.$$
We claim that $$ \abs{\{u^T(\cdot,0) > 1/2\}} \geq T^{1/4} \quad\text{ for $T$ large enough}. $$ Suppose not. Then,
using \eqref{secintur} we obtain $$ \displaystyle
\frac12\abs{\{u^T(\cdot,0) < -1/2\}} \leq  \Big| \displaystyle \int_{(-T,T)\cap \{u^T(\cdot,0) < -1/2\}}
 u^T(t,0)\, dt \Big|\le CT^{1/4}.
$$ Hence, all the three sets $\{|u^T(\cdot,0)| \leq 1/2\}$, $\{u^T(\cdot,0) > 1/2\}$, and $\{u^T(\cdot,0) < -1/2\}$
would have length smaller than $CT^{1/4}$. This is a contradiction for $T$ large, since these sets fill $(-T,T)$.

\medskip

\noindent {\it \large Step 3.}
To prove the third claim, we follow one of the three proofs proposed in  \cite{Cabre-SolaMorales} which is based in the stability of the minimizer and does not require sliding.

Since $u^T$ is an
 absolute minimizer, we have
\be Q(\xi)=\int_{Q_T^+}|\nabla_{\rr \times \rr_+} \xi|^2(\cosh t)^{n-1}y^a \;dt dy+\int_{-T}^Tf'(u^T(t,0))\xi^2(\cosh t)^{n-1}dt\geq 0\label{quadratic}\ee
for every $\xi \in W^{1,2}(Q_T^+)$ with $\xi\equiv 0 $ on $\partial Q_T^+\cap\{y>0\}$
 in the weak sense.

 We will justify that we can choose $\xi=(\partial_t u^T)^{-}$:
From  Lemma \ref{regularity1} we have that  $u^T\in W^{2,2}(Q_T^+)\cap \mathcal C^\beta(\overline{Q_T^+})$. Now, we would like to show that $ \partial_t u^T>0 $ on $\{t=T\}$ (and hence $(\partial_t u^T)^{-}\equiv 0$ on  $\partial Q_T^+\cap\{y>0\}$).
Note that $-H_a (u^T-v)<0$ in a neighborhood of $\{t=T\}$  for $T$ big enough. Since  $u^T-v\equiv 0$ on $\{t=T\}$,  using maximum principle for the difference $u^T-v$ and Hopf's boundary lemma from Section \ref{subsection:maximum} we have that $\partial_t(u^T-v)>0$ on $\{t=T\}$. Since  $v$ is increasing in $t$, we have that
$ \partial_t u^T>0 $ on $\{t=T\}$. Analogously, we can conclude on the whole $\partial^+ Q_T^+$.

Hence,  we can choose
 $\xi=(\partial_t u^T)^{-}$ as a test function in \equ{quadratic}. Integrating by parts we have that
$$Q(\xi)=\int_{-T}^T
\left(-y^a\partial_y|_{y=0} (\partial_t u^T)^- + f'(u^T(t,0)) (\partial_t u^T)^-\right) (\partial_t u^T)^- (\cosh t)^{n-1}dt\geq 0.$$
From the equation \eqref{equation100} satisfied by $u^T$, we must have the previous expression $Q(\xi)\equiv 0$. However, since we have strict stability, we  conclude that $(\partial_t u^T)^- \equiv 0$, as desired.

\medskip

\noindent {\it \large Step 4.}

The local convergence of $u^T$ in $\mathcal C^2$ follows from our a priori estimates from Section \ref{subsection-regularity}. Let $u$ be the local limit of $u^T$. We need to show that $u$
is not a trivial function. In order to prove that this is not the case we show that there is a  $t_0$ such that $u(t_0,0)=0$. Since $u$ would be  asymptotic to 1 or $-1$ as $t\to\infty$, it cannot be constant.

From Steps 2 and 3 we have that there is an $x_T$ that  satisfies
\begin{equation*} u^T(x_T,0) = 0.
\end{equation*}
Moreover, Step 2 implies that  $|x_T| \le T - T^{1/4} .$

In order to conclude, it is enough to  show that there is subsequence $x_T$ that converges to a finite value. To this end,  we follow the proof of Proposition 2.4 in  \cite{Birindelli-Mazzeo}, which is different from the Euclidean proof.

Suppose that $x_T\to +\infty$ and consider $w_T(t,y)=u^T(t+x_T, y)$. The functions $w_T$ satisfy
\begin{equation} \begin{cases} \partial_{tt} w_T+(n-1)\tanh (t+x_T) \partial_t w_T
+\frac{a}{y}\partial_y w_T+\partial_{yy} w_T
=0&\text{ in } \tilde{Q}_T^+,\\-y^a \partial_y w_T =f(w_T)&\text{
on }\partial^0 \tilde{Q}_T^+,\\   w_T=v &\text{ on }\partial^+ \tilde{Q}_T^+.\end{cases} \label{eqwt}\end{equation}
where $ \tilde{Q}_T^+ =(-T-x_T, T-x_T)\times(0,T^{\frac{1}{8}})$.
Note that although this is not the same equation we started with, the coefficients are uniformly bounded for every $T$. Hence, standard regularity theory implies that up to subsequence, $w_T(t,y)$ converges locally in $\mathcal C^2$ to a function $w$ that satisfies
\begin{equation*} \begin{cases} \partial_{tt} w+(n-1)\partial_t w
+\frac{a}{y}\partial_y w+\partial_{yy} w
=0&\text{ in } \rr\times \rr_+,\\-y^a \partial_y w =f(w)&\text{
on }\rr\times\{0\},\\   w\to \pm 1  &\text{ as }t\to\pm \infty.\end{cases} \end{equation*}
Additionally, we have that $w(0,0)=0$.  Multiplying  equation \eqref{eqwt} by $y^a w_T$ and integrating on $\tilde{Q}_T^+$ we have that
\bes\begin{split}
0=&\int_{\tilde{Q}_T^+} \left(\frac{y^a}{2}\partial_t\left[\left(\partial_t w_T\right)^2\right]
+(n-1)\tanh(t+x_T)y^a \left(\partial_t w_T\right)^2+ \partial_y(y^a \partial_y w_T)\partial_t w_T\right)\,dtdy
\\
=& \int_0^{T^{1/8}} \frac{y^a}{2}\left.\left(\partial_t w_T\right)^2\right|_{t=-T-x_T}^{T-x_T}dy
 + \int_{-T-x_T}^{T-x_T}\left(
T^{\frac{a}{8}} \partial_y w_T (t,T^{\frac{1}{8}})\partial_t w_T (t,T^{\frac{1}{8}})+
f(w_T(t,0))\partial_t w_T(t,0) \right)dt\\
&+\int_{\tilde Q_T^+} \left((n-1)y^a\tanh(t+x_T) \left(\partial_t w_T\right)^2-\frac{y^a}{2}\partial_t\left[\left(\partial_y w_T\right)\right]^2\right)\,dtdy\\
=& \int_0^{T^{1/8}} \frac{y^a}{2}\left.\left(\partial_t u^T\right)^2\right|_{t=-T}^Tdy
+(n-1)\int_{\tilde Q_T^+} y^a\tanh(t+x_T) \left(\partial_t w_T\right)^2 \,dtdy\\ &  + \int_{-T-x_T}^{T-x_T}
\mu T^{\frac{a}{8}} \partial_y w_T (t,T^{\frac{1}{8}})\left(\sech^2(\mu T)- \sech^2(-\mu T)\right)dt \\
& -F(\tanh(\mu T))+F(\tanh(-\mu T)).
\end{split}\ees

Taking $T\to \infty$ and considering that $\partial_t u_T\to 0 $ as $t\to \infty$ and $\partial_t w_T\to\partial_t w$  locally uniformly
we obtain
$$0=(n-1)\int_0^{\infty} \int_{-\infty}^{\infty} y^a \left(\partial_t w \right)^2\,dtdy, $$
 so $\partial_t w \equiv 0$ and $w$ is constant in $t$, but this is a contradiction with $w(0,0)=0$ and $w\to \pm 1 $ as $t\to\pm \infty$.

Hence, up to subsequence, we may assume that there exists $t_0$ satisfying $x_T\to t_0$. Then, we let $u^T\to u$ as $T\to \infty$
locally uniformly, \ $u(t_0,0)=0$ and
$\lim_{t\to \pm \infty} u(t,0)=\pm 1$. Moreover,  since $\partial_t u^T \geq 0$ in $Q_T^+$ we conclude that $\partial_t u(t,y) \geq 0$ for every $(t,y) \in \rr\times \rr_+$

We can show now that $u$ is a minimizer respect to compact perturbations: Consider
$\psi$ compactly supported on
$B_T^+\cup \Gamma^0_T$ and such that $|u+\psi|\le 1$ in $B_T^+$. Extend $\psi$
to be identically zero outside $B_T^+$, so that $\psi\in H^1_{\rm loc}(\overline{\R^2_+})$. Note that, since $-1<u<1$
and $-1\le u+\psi\le 1$, we have $-1<u+(1-\epsilon)\psi <1$ in $\overline{B_T^+}$ for every $0<\epsilon<1$. Hence, by
the local convergence of $\{u^T\}$ towards $u$, for $T$ large enough we have $B_T^+\subset S_T^+$ and $-1 \le
u^T+(1-\epsilon)\psi \le 1$ in $B_T^+$, and hence also in $S_T^+$. Then, since $u^T$ is a minimizer in $S_T^+$, we have
$E_{S_T^+}(u^T) \leq E_{S_T^+}(u^T + (1-\epsilon)\psi)$ for $T$ large. Since $\psi$ has support in $B_T^+\cup
\Gamma^0_T$, this is equivalent to $$ E_{B_T^+}(u^T) \leq E_{B_T^+}(u^T + (1-\epsilon)\psi) \quad \text{for $T$ large}.
$$ Letting $T \to \infty$, we deduce $E_{B_T^+}(u) \leq E_{B_T^+}(u +(1-\epsilon) \psi)$. We conclude now by letting
$\epsilon\to 0$.

Finally, since $\partial_t u\ge 0$, the limits $L^{\pm}=\lim_{t\to\pm\infty}u(t,0)$ exist. To establish that $u$ is a layer
solution, it remains only to prove that $L^\pm=\pm 1$. For this, note that Proposition~\ref{sameheight}  can be proved in the same fashion for one-dimensional solutions that
have finite (one-dimensional) reduced energy. We leave details to the reader. Applying this modified version of
 Proposition~\ref{sameheight} to
$u$, a local minimizer relative to perturbations in $[-1,1]$ we deduce that $$ F\ge F(L^-)=F(L^+) \quad\text {in }
[-1,1]. $$
Since in addition $F>F(-1)=F(1)$ in $(-1,1)$ by hypothesis, we infer that $|L^\pm|=1$. But $u(t_0,0)=0$ and
thus $u$ cannot be identically $1$ or $-1$. We conclude that $L^-=-1$ and $L^+=1$, and therefore $u$ is a layer solution.
\end{proof}

One could show that the layer solutions have exponential decay towards $\pm 1$. This is a much better behavior than the Euclidean case, that only has power decay at infinity. This is because of the metric factor $(\cosh t)^{n-1}$.

\subsection{Uniqueness}\label{section-uniqueness}

Here we prove the second statement in Theorem \ref{thm-one-dimensional} on uniqueness. As a byproduct, monotonicity is also shown (see Corollary \ref{cor-monotonicity}).

Uniqueness in Euclidean space follows from the sliding method. In the hyperbolic case, \emph{sliding} must be replaced by \emph{stretching}, a transformation that leaves the Laplacian \eqref{Laplace-Beltrami} invariant.  More precisely, when considering the half space model, we consider the isometry given by scaling by positive constants (see Section \ref{model} to recall the definition of this model). This method has been successfully employed to obtain symmetry results for Laplacian semilinear equations on hyperbolic space (see \cite{Kumaresan-Prajapat, Almeida-Ge,Birindelli-Mazzeo}).

For the fractional Laplacian in the Euclidean case, uniqueness is achieved by sliding on the horizontal variable $x$ (see \cite{Cabre-SolaMorales,Cabre-Sire:II}). In the following we show that this method still works in hyperbolic space if one `stretches' on the horizontal variable.

Note that the assumption on $F''(-1),F''(1)>0$ may be weakened to the requirement that $f$ is non-increasing in $(-1,\tau)\cup (\tau,1)$ for some $\tau\in(0,1)$.

Fix $\Pi$ a totally geodesic hyperplane, and $t$ the signed distance to $\Pi$. This hyperplane divides the unit sphere $\mathbb S^{n-1}=\partial_\infty \mathbb H^n$ into two regions $S_+$ y $S_-$. For any $x_0\in\Pi$, the variable $t$ gives the parametrization of a curve $\sigma$ in hyperbolic space which passes through $x_0$ and is orthogonal to $\Pi$. Let $P_+\in S_+$ and $P_-$ be the limits of such curve when $t\to +\infty$ and $t\to -\infty$, respectively.

Let $u_i=u_i(t,y)$, $i=1,2$, be two solutions of \eqref{1dimsol} such that
$$u(t,0)\to \pm 1 \hbox{ as } t\to \pm  \infty,\quad\text{and} \quad |u|< 1.$$

As mentioned above, we work with the upper half-space model of hyperbolic space. Compose with a M\"obius transform so that $P_-=0$ and $P_+=\infty$; the images of $S_-$ and $S_+$, which we denote by the same symbols, are then some ball in $\mathbb R^{n-1}$ containing 0, not necessarily at its center, and the exterior of this ball (union $\infty$), respectively. Moreover, the curve $\sigma$ is transformed into some ray $\Lambda$ in the upper half-space emanating from zero.

By some abuse of notation, denote $u_i(t,y)$, $i=1,2$, be the given solutions transplanted to the new model, in such a way that the $t$ coordinate is a parametrization of this ray, i.e., $\Lambda=\{t\in(0,\infty)\}$ (and notice that from now on, $t$ does not represent the signed distance, but an appropriate function of it). Assume also that both solutions are normalized in such a way that
\begin{equation}
\label{normalization}
u_1(1,0)=u_2(1,0)=0.
\end{equation}
For $R>1$ consider the rescaling
$$u_2^R(t)=u_2(Rt,y).$$
The important point is that the hyperbolic Laplacian \eqref{Laplace-Beltrami} is invariant under this type of transformations, so that $u_2^R$ is also a solution of \eqref{problem-extension}.

\begin{lemma}\label{lemma-rescaled}
For every $R>1$, $u_1\leq u_2^R$.
\end{lemma}

\begin{proof}
Let $v_R=u_2^R-u_1$, $v_R=v_R(t,y)$. It satisfies the equation
\begin{equation*} \begin{cases}
\partial_{yy} v_R+\frac{a}{y}\partial_y v_R + \Delta_{\mathbb H^n} v_R=0&\quad \text{in } \mathbb H^n\times \mathbb R_+,\\
-y^a \partial_y v_R|_{y=0}=d_R(v_R) &\quad \text{on }\mathbb H^n,
\end{cases} \end{equation*}
where
$$d_R(t)=\frac{f(u_2^R)-f(u_1)}{u_2^R-u_1}(t,0),$$
if $v_R(t,0)\neq 0$ and $d_R(t)=0$ otherwise. Note that $d_R$ is a bounded function since $f$ is Lipschitz.\\

\noindent \emph{Claim 1:} $v_R(t,y)>0$ for $R$ large enough.

By our hypothesis \eqref{eq4} on $u_i$, $i=1,2$, there exists a compact interval $[1/A,A]$ such that
$u_i(t,0)\in(-1,\tau)$ if $t\in(0,1/A]$ and $u_i(t,0)\in(\tau,1)$ if $t\in[A,\infty)$, for $i=1,2$. Take $R>0$ sufficiently large such that $v_R(t,0)>0$ for $t\in[1/A,A]$.  Setting
$$\Sigma:=[1/A,A]\cup\{t\in(0,\infty) : v_R(t,0)>0\},$$
the claim is proved thanks to Lemma \ref{lemma-max1}, applied in the new coordinates.\\
\noindent\emph{Claim 2:} Assume that $u_1\equiv u_2$. If $R>1$ and $v_R\geq 0$, then $v_R\not\equiv 0$.

To show this, suppose that there exists $R>1$ such that $v_R\equiv 0$, i.e., $u_1(Rt)=u_1(t)$ for all $t\in(0,\infty)$. This is a contradiction with the fact that
 $$\lim_{t\to 0} u_1(t,0)=-1\quad \text{and}\quad \lim_{t\to +\infty} u_1(t,0)=+1.$$
 \noindent\emph{Claim 3:} Assume that $u_1\equiv u_2$. If $v_R\geq 0$ for some $R>1$, then $v_{R+\mu}\geq 0$ for every $\mu$ small enough (with $\mu$ either positive or negative).

By Hopf's maximum principle from Lemma \ref{lemma-max} and the previous claim, we must have $v_R>0$. Let $K_R$ be a compact interval in $(0,\infty)$ such that, for $t\not\in K_R$, $|u_1(t,0)|>1-\tau/2$ and $|u_1^R(t,0)|>1-\tau/2$. By continuity and the existence of limits, we have that if $|\mu|$ is small enough, then $v_{R+\mu}(t,0)>0$ for $t\in K_R$ and $|u_1^{R+\mu}|>1-\tau$. Hence we can apply
Lemma \ref{lemma-max1} again  to $v_{R+\mu}$ and $\Sigma=K_R$ in order to prove the claim.\\
\noindent\emph{Claim 4:} Proof of Lemma \ref{lemma-rescaled} is completed in the case $u_1\equiv u_2$.

Indeed, these three claims that  $\{R>1 : v_R\geq 0\}$ is a nonempty, closed and open set in $(1,\infty)$, and hence equal to the whole interval. This completes the proof of the lemma. In addition, we obtain that for every solution $u_1$,
$$0\leq\left.\frac{d}{dR}\right|_{R=1} u_1(Rt,y)=t\partial_t u(t,y),$$
which gives that $u_1(t,y)$ is non-decreasing in $t$.
By the strong maximum principle, $u_1(t,y)$ is also strictly increasing in $t$.\\

\noindent\emph{Claim 5:} Lemma \ref{lemma-rescaled} is also true when $u_1\not\equiv u_2$.

Indeed, we just need to reprove Claim 2 and the rest of the claims will follow similarly. Assume, by contradiction, that $u_1\equiv u_2^R$ for some $R>1$. Then
$$u_1\lp\tfrac{1}{R},0\rp=u_2^R\lp\tfrac{1}{R},0\rp=u_2(1,0)=0$$
by our normalization \eqref{normalization}. This gives that both points $\lp\frac{1}{R},0\rp$ and $(1,0)$ are zeroes of $u_1$. Contradiction with the fact that $u_1$ is strictly increasing in $t$.

\end{proof}

An important consequence of the Lemma is the following monotonicity result:

\begin{cor}
\label{cor-monotonicity}
Assume that $u(t,y)$ is a solution of \eqref{1dimsol} satisfying \eqref{eq3}-\eqref{eq4}. Then $u$ is increasing in $t$ for every fixed $y$.
\end{cor}

\begin{cor}
If $u_1$ and $u_2$ are solutions of \eqref{1dimsol} satisfying \eqref{eq3}-\eqref{eq4}, then they must coincide up to stretching. In particular, there is a unique solution that satisfies $u(0,0)=0$.
\end{cor}

\begin{proof}
First, we may rescale both $u_1,u_2$ so that they satisfy the normalization condition \eqref{normalization}. Then take $R\to 1$ in Lemma \ref{lemma-rescaled}.
\end{proof}

\section{Two-dimensional symmetry}

In this section we provide the proof of Theorem \ref{thm-symmetry}. The geometric method was developed in \cite{Birindelli-Mazzeo}, and here we adapt it for the fractional case adding the $y$ variable. Let $P_-\in S_-$ and $P_+\in S_+$ be arbitrary points. Let us show first that  $u$ is non-decreasing along any curve of constant geodesic curvature joining $P_-$ to $P_+$. Then it is easy to show that $u$ is a function of only two variables. \\

\noindent \emph{Claim 1: } Let  $\sigma(t)$ be any curve of constant geodesic curvature in $\mathbb H^n$ such that $\lim_{t\to\pm \infty} \sigma(t)=P_{\pm}$. Then $u(\sigma(t),y)$ is non-decreasing in $t$ for fixed $y$.

To prove the claim, we work on the upper half-space model as in the previous section. Compose with a M\"obius transformation so that $P_-=0$ and $P_+=\infty$. The images of $S_-$ and $S_+$  are then some ball in $\mathbb R^{n-1}$ containing 0, but not necessarily at its center, and the exterior of this ball (union $\infty$), respectively. Moreover, the curve $\sigma$ is transformed into some ray $\Lambda$ in the upper half-space emanating from $0$.

Fix $\tau>0$ such that $f'(s)>0$ when $s\in[-1,-1+\tau)\cup (1-\tau,1]$ and choose $A\in(0,1)$ so that $u(x)<-1+\tau$ for $x\in D_-(A):=\{x\in \mathbb H^n : |x|<A\}$ and $u(x)>1-\tau$ for $x\in D_+(A):=\{x\in \mathbb H^n : |x|>1/A\}$.

Consider the rescaling $u_R(x)=u(Rx)$. The rest of the proof follows the same lines as in Lemma~\ref{lemma-rescaled} and Corollary~\ref{cor-monotonicity}. \\

\noindent \emph{Claim 2:} $u$ is a function of just two variables.

Working again on the ball model, suppose that $S_-$ and $S_+$ are the lower and upper hemispheres of the boundary, respectively. Let $P$ be any point in the interior of the ball, and let $\Pi$ be the spherical cap which passes through $P$ and $\overline{S_-}\cap \overline{S_+}$. Let $\pi$ be any two-dimensional plane passing through the origin of the ball and the point $P$. Then $\pi\cap \Pi$ is a curve $\gamma$ of constant geodesic curvature in $\mathbb H^n$ passing through $P$ and limiting on two points $Q,Q'\in S_-\cap S_+$. It is easy to see geometrically, that we can approximate $\gamma$ by two sequences of curves of constant geodesic curvature $\gamma_j^-(t)$ and $\gamma_j^+(t)$ such that $\gamma_j^{\pm}(0)=P$ for all $j$,
$$\lim_{t\in -\infty} \gamma_j^-(t),\lim_{t\to +\infty}\gamma_j^+(t)\in S_-,\quad
\lim_{t\in +\infty} \gamma_j^-(t),\lim_{t\to -\infty}\gamma_j^+(t)\in S_+,$$
and
$$T_P \Pi\ni X=\lim_{j\to\infty}(\gamma_j^-)'(0)=-\lim_{j\to\infty} (\gamma_j^+)'(0).$$
Since $u(\gamma_j^-(t))$ and $u(\gamma_j^+(t))$ are both nondecreasing by the previous claim, we see that $\nabla u_P \cdot X=0$. However, $X$ can be chosen arbitrarily in $T_P \Pi$, which shows that $\nabla u(P)$ is orthogonal to $\Pi$.

We have now proved that if $\{\Pi_t\}$ is the foliation of $\mathbb H^n$ by hypersurfaces which are of (signed) distance $t$ from the totally geodesic copy of $\mathbb H^{n-1}$ with boundary $S_-\cap S_+$, then each $\Pi_t$ is a level set of $u$. In other words, $u$ is a function of the distance $t$ and the coordinate $y$ alone, as desired.\\

Once we know that $u$ only depends on the $t$ variable in the horizontal direction, then it must precisely be the one found in Theorem \ref{thm-one-dimensional}. Proof of Theorem \ref{thm-symmetry} is completed.

\qed


\section{Hamiltonian estimates}\label{section-Hamiltonian}

Here we show that there exists a Hamiltonian quantity that, although it does not remain constant along the trajectories as in the Euclidean case, it does decrease to zero when $t\to +\infty$. Let $u$ be a solution of
\begin{equation}\label{problem2}
\begin{cases}
H_a u = 0 &\quad\hbox{for } (t,y)\in \mathbb R\times \rr_+,\\
-d_\gamma y^a \partial_y u |_{y=0}=f(u) &\quad \hbox{for } y\in \mathbb R,
\end{cases}
\end{equation}
where the operator $H_a$ is defined in \eqref{operatorH} and the constant $d_\gamma$ in \eqref{constant-d}.
Let
\be\label{V}V(t)=\frac{1}{2}\int_0^\infty y^a \lp (\partial_t u)^2-(\partial_y u)^2\rp\,dy-\frac{1}{d_\gamma}\lp F(u(t,0))-F(1)\rp.\ee
Differentiating in the variable $t$
$$V'(t)=\int_0^\infty y^a (\partial_t u \partial_{tt} u-u_y u_{yt})dy+\frac{1}{d_\gamma} f(u(t,0)) u_t(t,0).$$
After integration by parts, taking into account the second equation in \eqref{problem2} and the decay at infinity for the boundary terms, we get that
$$V'(t)=\int_0^\infty y^a \partial_t u \partial_{tt} u+(y^a \partial_y u)_y \partial_t u\,dy.$$
Next, the first equation in \eqref{problem2} allows to rewrite
\be\label{V'}V'(t)=-(n-1)\int_0^\infty y^a \tanh t (\partial_t u)^2 \,dy.\ee
In particular, this shows that $V$ is decreasing when $t>0$ and increasing when $t<0$, attaining the maximum at $t=0$. On the other hand, $V(+\infty)=V(-\infty)=0$.

 After some calculations, one may check that
 $$\lp\cosh^{n-1}t V(t)\rp_t=(\cosh^{n-1}t)_t \left[-\frac{1}{2}\int_0^\infty y^a [(\partial_t u)^2+(\partial_y u)^2]\,dy-\frac{1}{d_\gamma} \lp F(u(t,0))-F(1)\rp\right].$$

We have proved:

\begin{prop}
The Hamiltonian energy \eqref{V} is decreasing to zero along the trajectories when $t\to +\infty$.
\end{prop}

\section{Convergence of the layers when $\gamma\to 1$}

Here we prove Theorem \ref{thm-limit}. For simplicity of the notation, we drop the subindex $k$ and just denote the sequence by $\{w_\gamma\}$ when $\gamma\to 1$. Let $u_\gamma$ be the extension of $w_\gamma$ to $\mathbb H^n\times \mathbb R^+$, i.e., $u_\gamma$ is the solution of
\begin{equation*}\label{problem1}
\begin{cases}
\partial_{yy} u_\gamma+\frac{a}{y}\partial_y u_\gamma+\Delta_{\mathbb H^n} u_\gamma = 0 &\quad\hbox{for } (x,y)\in \hh^n\times \rr_+,\\
-d_\gamma y^a \partial_y u_\gamma |_{y=0}=f(u_\gamma)& \quad \hbox{for } x\in \hh^n,
\end{cases}
\end{equation*}
Note that we have been very careful with the multiplicative constant  in front of the nonlinearity, whose value is precisely given in \eqref{constant-d}.
As it was shown in \cite{Cabre-Sire:I}, for $a=1-2\gamma$,
\be\label{limits-gamma}\frac{d_\gamma}{(1-a)^{-1}}\to 1 \quad\text{as}\quad\gamma\downarrow 0\qquad \text{and}
\qquad \frac{d_\gamma}{1+a}\to 1 \quad\text{as}\quad\gamma\uparrow 1.\ee

The existence of a limit $w_\gamma\to\overline w$ follows exactly the arguments in Section 6 of \cite{Cabre-Sire:I} using our uniform estimates from Section~\ref{subsection-regularity} and we will refer the reader to that paper. By appropriate stretching, we may assume that $w_\gamma(0)=0$.   It is clear from the arguments  of \cite{Cabre-Sire:I}  that $\overline w$ satisfies equation \eqref{layer1},
$$\overline w(0)=0\quad \text{and} \quad \overline w'\geq 0.$$
As a consequence, the function $\overline w$ admits limits at $\pm\infty$,
$$\lim_{t\to\pm\infty} \overline w(t)=L^{\pm}\in[-1,1].$$

Now we need to prove that $\overline w$ is indeed a layer, i.e., $L^\pm=\pm 1$. In \cite{Cabre-Sire:I} the authors use a very sharp Hamiltonian estimate. However, we have found that is enough to have the results from Section \ref{section-Hamiltonian}. In view of \eqref{Laplacian1}, $\overline w$ is a solution of
$$-\partial_{tt} \overline w -(n-1)\tanh t \,\partial_t \overline w=f(\overline w).$$
Multiply the above equation by $\partial_t \overline w$ and integrate it over the interval $(t,\infty)$. We obtain
\be\label{formula4}\frac{1}{2}(\partial_t \overline w) ^2-(n-1)\int_t^\infty (\tanh s)\,(\partial_s\overline w)^2 ds=-F(L^+)+F(\overline w(t)).\ee
On the other hand, by the passage to the limit justified in \cite{Cabre-Sire:I},
\be\label{formula1}\lim_{\gamma\to 1} (1+a)\int_0^\infty y^a (\partial_t u_\gamma)^2 dy=(\partial_t \overline w)^2.\ee
For the same reason, and using expression \eqref{V'},
\begin{equation}\label{formula2}\begin{split}
\lim_{\gamma\to 1} (1+a)V_\gamma(t)&=-\lim_{\gamma\to 1}(1+a)\int_t^\infty V_\gamma'(s)ds\\
&=(n-1)(1+a)\lim_{\gamma\to 1}\int_0^\infty \int_t^\infty y^a (\tanh s) (\partial_s u_\gamma)^2\,ds\,dy \\
& = (n-1)\int_t^\infty (\tanh s) (\partial_s \overline w)^2\,ds
\end{split}\end{equation}
Next, from formula \eqref{V} we deduce that
\begin{equation*}\begin{split}
\frac{1}{2}\int_0^\infty y^a (\partial_t u_\gamma)^2 dy-V_\gamma(t)&=\frac{1}{2}\int_0^\infty y^a(\partial_y u_\gamma)^2\,dy+\frac{1}{d_\gamma} \left[F(u_\gamma(t,0))-F(1)\right] \\
&\geq \frac{1}{d_\gamma} \left[F(u_\gamma(t,0))-F(1)\right]
\end{split}\end{equation*}
Therefore, passing to the limit $\gamma\to 1$ in the previous expression, and substituting \eqref{formula1} and \eqref{formula2}, we arrive at
\begin{equation}\label{formula3}
\frac{1}{2}(\partial_t\overline w)^2-\frac{(n-1)}{1+a}\int_t^\infty (\tanh s)(\partial_s\overline w)^2 \geq F(\overline w(t))-F(1),
\end{equation}
where we have also used the asymptotic behavior \eqref{limits-gamma}. Putting together expressions \eqref{formula4} and \eqref{formula3} we conclude that
$$F(L^+)\leq F(1).$$
Because $L^+\geq 0$ and the initial hypothesis on our double well potential $F$, we must have that $L^+=1$, as desired. In the same way, we prove that $L^-=-1$. Hence, $\overline w$ is the layer solution connecting $-1$ to $1$ with $\overline w(0)=0$. Uniqueness of this $\overline w$ follows from \cite{Birindelli-Mazzeo}.

\section{Multilayer solutions}

Here we provide the proof of Theorem \ref{thm-multilayer}. First remark that \cite{Mazzeo-Saez} gives the construction of a multilayer solution in the $\gamma=1$ case, call it $u_1$. For $\gamma$ close to one, we use a perturbation argument in the exponent $\gamma$  that was introduced in \cite{Gonzalez-Mazzeo-Sire}.

For each $\Pi_j$, choose a hyperbolic isometry $\varphi_j$ which carries $\Pi_j$ to a fixed totally
geodesic hyperplane $\Pi$. We consider weighted H\"older spaces
\[
\calC^{k,\alpha}_{\mu,\delta} (\HH^n,\calH) := \sech(\mu \tau) \rho^\delta \calC^{k,\alpha}(\HH^n) =
\{ u = \sech(\mu \tau) \rho^\delta \tilde{u}: \tilde{u} \in \calC^{k,\alpha}(\HH^n) \},
\]
where the function $\tau$ is a  smoothing of the signed distance function from the union of the hyperplanes $\Pi_j$ and $\rho$ is defined below.

 Consider the one layer case $\Pi$. Let $\rho_0$ be a boundary defining function for $\Pi$, this function is strictly positive on $\overline{\mathbb H^n} \backslash (\mathbb S^{n-1}\cap \overline \Pi)$.  Let $\hat{\chi}$ be a smooth nonnegative cutoff function which equals $1$ on a neighborhood of $\overline{\Pi} \subset
\overline{\HH^n}$ and which vanishes outside a slightly larger neighborhood.
 Then we take the  function $\rho$ as
 \[
\rho = \sum_{j=1}^N \varphi_j^*( \hat{\chi} \rho_0 ) + \sum_{j=1}^N \varphi_j^*( 1 - \hat{\chi});
\]
it agrees with the pullback $\varphi_j^* (\hat{\chi} \rho_0)$ near $\overline{\Pi_j}$ and is strictly positive elsewhere on
the closure of $\HH^n$. For further details on these definitions we refer to \cite{Mazzeo-Saez}.

We consider now, for each $\gamma\in(0,1]$, the linearization of problem \eqref{initial-problem}\quad  $u \mapsto (-\Delta_\HH)^\gamma u - f(u)$ around the solution ${u_1}$ for the $\gamma=1$ case. It is given by the operator
\[
v \mapsto L_\gamma v := (-\Delta_\HH)^\gamma v - f'({u_1})v.
\]
Define \[
- \beta_\pm = -\frac{n-1}{2} - \sqrt{ \frac{(n-1)^2}{4} + f'(\pm1)}\, .
\]
and $\beta=\min\{\beta_+, \beta_-\}$.
 It is proved  \cite{Mazzeo-Saez} that the mapping
\[
L_1 : \calC^{2,\alpha}_{\mu, \delta}(\HH^n, \calH) \longrightarrow \calC^{0,\alpha}_{\mu, \delta}(\HH^n,\calH)
\]
is surjective for $\mu \in (0, \beta)$ and $\delta\in \left(0,\frac{n-2}{2}\right)$.

We claim that for $\gamma$ sufficiently close to $1$, and for $\mu \in (0,\beta)$, where $\beta > 0$
is some small fixed number, the mapping
\[L_\gamma: \calC^{2,\alpha}_{\mu, \delta}(\HH^n, \calH) \longrightarrow \calC^{0,\alpha+2(1-\gamma)}_{\mu-2\gamma,  \delta}(\HH^n,\calH)\]
is also bounded and surjective.  The assertion about the boundedness of $L_\gamma$
is clearly true for $\gamma = 0, 1$, and hence by interpolation is true for all $\gamma$ close to $1$.
Note that $L_\gamma$ is a pseudodifferential edge operator of order $2\gamma$. Then from \cite{Mazzeo:edge-operators}, one can see that $L_\gamma$ is Fredholm,
and since it is surjective at $\gamma = 1$, it must remain surjective for values of $\gamma$ which are close to $1$. We write its right inverse as $G_\gamma$.

Next, consider the mapping
\[
(\gamma, v) \longmapsto N(\gamma,v) := G_\gamma
[(-\Delta_{\mathbb H^n})^\gamma (u_1 + v) - f(u_1+v)].
\]
It is clear that $N(1,0) = 0$.
Let $v$ lie in a ball of radius $\e$ about $0$ in the space $\mathcal C^{2,\alpha}_{\mu,\delta}$. Clearly $\left. D_v N\right|_{(1,0)} = G_1 L_1 = \mbox{Id}$. The implicit function theorem now
applies to show that for every $\gamma$ near to $1$, there exists a unique $v_\gamma \in \mathcal C^{2,\alpha}_{\mu,\delta}$
with norm less than $\e$ such that
$u_\gamma = u_1+ v_\gamma$ is a solution of our problem. The proof of Theorem \ref{thm-multilayer} is completed.\\


\textbf{Acknowledgements:} The authors acknowledge the hospitality of Universitat Polit\`ecnica de Catalunya where part of this work was carried out.


\end{document}